\documentclass[10pt]{amsart}
\usepackage{amssymb,MnSymbol}
\usepackage{amsthm,amsmath}
\usepackage{cite}

\title[Relaxations of associativity and preassociativity]{Relaxations of associativity and preassociativity for variadic functions}

\author{Miguel Couceiro}
\address{LORIA \\ (CNRS - Inria Nancy Grand Est - Universit\'e de Lorraine)\\
BP239~-~ 54506 Vandoeuvre-l\`es-Nancy, France} \email{miguel.couceiro[at]inria.fr }

\author{Jean-Luc Marichal}
\address{Mathematics Research Unit, FSTC, University of Luxembourg \\
6, rue Coudenhove-Kalergi, L-1359 Luxembourg, Luxembourg} \email{jean-luc.marichal[at]uni.lu }

\author{Bruno Teheux}
\address{Mathematics Research Unit, FSTC, University of Luxembourg \\
6, rue Coudenhove-Kalergi, L-1359 Luxembourg, Luxembourg} \email{bruno.teheux[at]uni.lu }

\date{August 4, 2015}

\begin{document}

\theoremstyle{plain}
\newtheorem{theorem}{Theorem}[section]
\newtheorem{lemma}[theorem]{Lemma}
\newtheorem{proposition}[theorem]{Proposition}
\newtheorem{corollary}[theorem]{Corollary}
\newtheorem{fact}[theorem]{Fact}
\newtheorem*{main}{Main Theorem}

\theoremstyle{definition}
\newtheorem{definition}[theorem]{Definition}
\newtheorem{example}[theorem]{Example}
\newtheorem{problem}[theorem]{Problem}

\theoremstyle{remark}
\newtheorem*{conjecture}{Conjecture}
\newtheorem{remark}{Remark}
\newtheorem{claim}{Claim}

\newcommand{\N}{\mathbb{N}}
\newcommand{\Q}{\mathbb{Q}}
\newcommand{\R}{\mathbb{R}}

\newcommand{\ran}{\mathrm{ran}}
\newcommand{\dom}{\mathrm{dom}}
\newcommand{\id}{\mathrm{id}}
\newcommand{\med}{\mathrm{med}}
\newcommand{\Ast}{\boldsymbol{\ast}}
\newcommand{\Cdot}{\boldsymbol{\cdot}}

\newcommand{\bfu}{\mathbf{u}}
\newcommand{\bfv}{\mathbf{v}}
\newcommand{\bfw}{\mathbf{w}}
\newcommand{\bfx}{\mathbf{x}}
\newcommand{\bfy}{\mathbf{y}}
\newcommand{\bfz}{\mathbf{z}}

\newcommand{\calA}{\mathcal{A}}
\newcommand{\calP}{\mathcal{P}}
\newcommand{\calI}{\mathcal{I}}

\begin{abstract}
In this paper we consider two properties of variadic functions, namely associativity and preassociativity, that are pertaining to several data and language processing tasks. We propose parameterized relaxations of these properties and provide their descriptions in terms of factorization results. We also give an example where these parameterized notions give rise to natural hierarchies of functions and indicate their potential use in measuring the degrees of associativeness and preassociativeness. We illustrate these results by several examples and constructions and discuss some open problems that lead to further directions of research.
\end{abstract}



\maketitle

\section{Introduction}

Let $X$ be an arbitrary nonempty set, called the \emph{alphabet}, and its elements are called \emph{letters}. The symbol $X^*$ stands for the set $\bigcup_{n \geqslant 0} X^n$ of all tuples on $X$, and its elements are called \emph{strings}, where the empty string $\varepsilon$ is such that $X^0=\{\varepsilon\}$. We denote the elements of $X^*$ by bold roman letters $\bfx$, $\bfy$, $\bfz$, $\ldots$ If we want to stress that such an element is a letter of $X$, we use non-bold italic letters $x$, $y$, $z$, $\ldots$ We assume that $X^*$ is endowed with the concatenation operation (the empty string $\varepsilon$ being the neutral element) for which we adopt the juxtaposition notation.  For instance, if $\bfx\in X^m$ and $y\in X$, then $\bfx y\varepsilon=\bfx y\in X^{m+1}$. For every string $\bfx$ and every integer $n\geqslant 0$, the power $\bfx^n$ stands for the string obtained by concatenating $n$ copies of $\bfx$. In particular, we have $\bfx^0=\varepsilon$. The \emph{length} of a string $\bfx$ is denoted by $|\bfx|$. In particular, we have $|\varepsilon|=0$.

Let $Y$ be a nonempty set. Recall that a function $F\colon X^*\to Y$ is said to be \emph{variadic} and that, for every integer $n\geqslant 0$, a function $F\colon X^n\to Y$ is said to be \emph{$n$-ary}. A unary operation on $X^*$ is a particular variadic function $F \colon X^* \to X^*$ called a \textit{string function} over the alphabet $X$.

\begin{definition}\label{de:AssStr}
A function $F\colon X^*\to X^*$ is said to be \emph{associative} \cite{LehMarTeh14} if for any $\bfx,\bfy,\bfz \in X^*$, we have
\begin{equation}\label{eq:assoc}
F(\bfx \bfy \bfz) ~=~ F(\bfx F(\bfy) \bfz){\,}.
\end{equation}
A function $F\colon X^*\to Y$ is said to be \emph{preassociative} \cite{MarTeh14,MarTeh14b} if for any $\bfx,\bfy,\bfy',\bfz \in X^*$, we have
\begin{equation}\label{eq:preassoc}
F(\bfy) ~=~ F(\bfy')\quad\Rightarrow\quad F(\bfx\bfy\bfz) ~=~ F(\bfx\bfy'\bfz){\,}.
\end{equation}
\end{definition}

Associative string functions and preassociative variadic functions as well as some of their variants have been studied in \cite{LehMarTeh14,MarMatTom,MarTeh14,MarTeh14b,MarTeh15,MarTeh15b}. For instance, it has been shown \cite{LehMarTeh14} that a function $F\colon X^*\to X^*$ is associative if and only if it is preassociative and satisfies the condition $F=F\circ F$. Also, under the Axiom of Choice, a function $F\colon X^*\to Y$ is preassociative if and only if it can be written as a composition of the form $F=f\circ H$, where $H\colon X^*\to X^*$ is associative and $f\colon\ran(H)\to Y$ is one-to-one.

It is noteworthy that several data processing tasks correspond to associative and preassociative functions. For instance, the function which corresponds to sorting the letters of every string in alphabetical order is associative. Similarly, the function that transforms a string of letters into upper case is also associative. Another natural example of a preassociative function is the mapping that outputs the length of strings.

In this paper we introduce and study certain relaxations of associativity and preassociativity. Let $\calA$ denote the class of associative string functions on $X^*$ and let $\calP$ denote the class of preassociative variadic functions on $X^*$. For a fixed nonempty subset $D$ of $X^*$, define the following classes of functions:
\begin{eqnarray*}
\calA_D &=& \{F\colon X^*\to\ran(F)\mid\text{$F(D)\subseteq X^*$ and \eqref{eq:assoc} holds for all $\bfx,\bfy,\bfz \in X^*$ such that $\bfy\in D$}\},\\
\calA'_D &=& \{F\colon X^*\to\ran(F)\mid\text{$F(D)\subseteq X^*$ and \eqref{eq:assoc} holds for all $\bfx,\bfy,\bfz \in X^*$ such that $\textstyle{F(\bfy)\in F(D)}$}\},\\
\calP_D &=& \{F\colon X^*\to\ran(F)\mid\text{\eqref{eq:preassoc} holds for all $\bfx,\bfy,\bfy',\bfz \in X^*$ such that $\bfy,\bfy'\in D$}\},\\
\calP'_D &=& \{F\colon X^*\to\ran(F)\mid\text{\eqref{eq:preassoc} holds for all $\bfx,\bfy,\bfy',\bfz \in X^*$ such that $\bfy\in D$}\}.
\end{eqnarray*}

It is clear that $\calA_{X^*}=\calA'_{X^*}=\calA$ and $\calP_{X^*}=\calP'_{X^*}=\calP$. When $D\varsubsetneq X^*$, these classes of functions correspond to relaxations of associativity and preassociativity for which we have $\calA'_D\subseteq\calA_D$ and $\calP'_D\subseteq\calP_D$. For instance, functions $F\colon X^*\to X^*$ that are in $\calA'_D$ are characterized by the fact that for any $\bfx,\bfy,\bfz\in X^*$ the value $F(\bfx\bfy\bfz)$ can be replaced with $F(\bfx F(\bfy)\bfz)$ whenever $F(\bfy)=F(\bfy')$ for some $\bfy'\in D$.

Certain of these relaxations are particularly natural. For instance, consider the subset
$$
D ~=~ \{x^n\mid x\in X,~n\in\N\}{\,},
$$
where $\N$ denotes the set of nonnegative integers. Any function $F\colon X^*\to X^*$ in $\calA_D$ has the property that the value $F(\bfx\bfy\bfz)$ can be replaced with $F(\bfx F(\bfy)\bfz)$ whenever $\bfy$ is a repeated letter. Further examples include:
\begin{itemize}
\item $D=\{\bfx\in X^*\mid ~|\bfx|\leqslant m\}$ for some integer $m\geqslant 0$,
\item $D=\{\bfx\in X^*\mid ~|\bfx|\geqslant m\}$ for some integer $m\geqslant 0$,
\item $D=X^*\bfw X^*=\{\bfx\bfw\bfx'\mid\bfx,\bfx'\in X^*\}$ for a given $\bfw\in X^*$,
\item $D=\{x\in\R\mid ~ x\leqslant s\}$ for some threshold $s\in\R$ (observe that $D\varsubsetneq\R\varsubsetneq\R^*$).
\end{itemize}

The function classes defined above can be motivated by indexation techniques in natural language processing (NLP) as they include noteworthy examples such as the Soundex encoding and its variants (see, e.g., \cite{JurMar08,Knu98}).

\begin{example}\label{ex:7sfd5}
Let $X=\{a,b,\ldots,z\}$, let $\bfw\in X^*$, and let $D=X^*\bfw X^*$. Consider also $F\colon X^*\to X^*$ defined by $F(\bfx)=\bfw$ if $\bfx\in X^*\bfw X^*$, and $F(\bfx)=\varepsilon$, otherwise. It is easy to see that $F$ is in $\calA'_D$. However, it is not in $\calA$ unless $|\bfw|\leqslant 1$. For example, if $\bfw = ab$, then $F(ab)=ab\neq\varepsilon =F(F(a)b)$.
\end{example}



\begin{fact}\label{fact:12}
For any nonempty subsets $D_1$ and $D_2$ of $X^*$ such that $D_1\subseteq D_2$, the following inclusions hold:
\begin{eqnarray*}
\calA_{D_2} ~\subseteq ~ \calA_{D_1}{\,},&&\quad \calA'_{D_2} ~\subseteq ~ \calA'_{D_1}{\,},\\
\calP_{D_2} ~\subseteq ~ \calP_{D_1}{\,},&&\quad \calP'_{D_2} ~\subseteq ~ \calP'_{D_1}{\,}.
\end{eqnarray*}
\end{fact}

\begin{problem}
Give necessary and sufficient conditions on $D$ for the inclusions $\calA'_D\subseteq\calA_D$ and $\calP'_D\subseteq\calP_D$ to be strict, and similarly for the inclusions in Fact~\ref{fact:12}.
\end{problem}

The outline of this paper is as follows. In Section 2 we focus on the special case when $D$ is the set of the strings over $X$ whose lengths are bounded above by a fixed integer $m\geqslant 0$. We describe a couple of examples (Examples \ref{ex:2.ex1}, \ref{ex:2.ex2}, \ref{ex:e675}), which show that in this case the inclusions given in Fact~\ref{fact:12} are strict, thus giving rise to hierarchies of nested classes of functions.
In Section 3 we present several factorization results. In particular, we identify associative functions within the class of preassociative functions, and extend these results to classes of `range-determined' functions in Section 4. The potential use of hierarchies in measuring associativeness and preassociativeness is then illustrated in Section 5, together with some open problems that constitute topics of current research. Other noteworthy questions are mentioned throughout the paper.

We use the following notation. The set $\{0,1,2,\ldots\}$ of nonnegative integers is denoted by $\N$. The domain and range of any function $f$ are denoted by $\dom(f)$ and $\ran(f)$, respectively. The identity function on any nonempty set $E$ is denoted by $\id_E$. For any function $F\colon X^*\to Y$ and any integer $n\geqslant 0$, we denote by $F_n$ the \emph{$n$-ary part} of $F$, i.e., the restriction $F|_{X^n}$ of $F$ to the set $X^n$.

Recall that a function $g$ is a \emph{quasi-inverse} \cite[Sect.~2.1]{SchSkl83} of a function $f$ if
$$
f\circ g|_{\ran(f)}=\id_{\ran(f)}\qquad\mbox{and}\qquad\ran(g|_{\ran(f)})=\ran(g).
$$
In this case we have $\ran(g)\subseteq\dom(f)$ and the function $g|_{\ran(f)}$ is one-to-one. Denote the set of quasi-inverses of a function $f$ by $Q(f)$. Under the Axiom of Choice (AC), the set $Q(f)$ is nonempty for any function $f$. In fact, AC is just another form of the statement ``every function has a quasi-inverse''. Note also that the relation of being quasi-inverse is symmetric: if $g \in Q(f)$ then $f \in Q(g)$.

\section{The case of bounded strings}

In this section we consider the special case when the set $D$ consists of strings whose lengths are bounded above by a given integer $m\geqslant 0$. Denote this set by $D_m$, i.e.,
$$
D_m ~=~ \{\bfx\in X^*\mid ~|\bfx|\leqslant m\}.
$$

From Fact \ref{fact:12} we immediately derive the inclusions
\begin{eqnarray*}
\calA_{D_{m+1}} ~\subseteq ~ \calA_{D_m}{\,},&&\quad \calA'_{D_{m+1}} ~\subseteq ~ \calA'_{D_m} ~\subseteq ~\calA_{D_m}{\,},\label{eq:75f1}\\
\calP_{D_{m+1}} ~\subseteq ~ \calP_{D_m}{\,},&&\quad \calP'_{D_{m+1}} ~\subseteq ~ \calP'_{D_m} ~\subseteq ~\calP_{D_m}{\,},\label{eq:75f2}
\end{eqnarray*}
as well as the equalities
$$
\calA ~=~ \bigcap_{m\geqslant 0}\calA'_{D_m} ~=~ \bigcap_{m\geqslant 0}\calA_{D_m}\quad\text{and}\quad \calP ~=~ \bigcap_{m\geqslant 0}\calP'_{D_m} ~=~ \bigcap_{m\geqslant 0}\calP_{D_m}.
$$

We now prove that each of these inclusions is actually strict, thus showing that these classes give rise to hierarchies of supersets of associative and preassociative functions.

Let $m\geqslant 0$ be an integer. We observe that any function $F\colon X^*\to X^*$ such that $F_k=\id_{X^k}$ for $k=0,\ldots,m$ is necessarily in $\calA_{D_m}$. However, the converse does not hold. For instance, the function $F\colon\N^*\to\N\cup\{\varepsilon\}$ defined by $F(\varepsilon)=\varepsilon$ and $F(\bfx)=|\bfx|$ for every $\bfx\in\N^*\setminus\{\varepsilon\}$ is in $\calA_{D_1}$ and its unary part $F_1=1$ is constant.

More generally, we also observe that any function $F\colon X^*\to X^*$ such that $F_k=\id_{X^k}$ for $k=0,\ldots,m$ and that satisfies the condition
\begin{equation}\label{eq:inter4}
F(\bfy) ~\in ~ \bigcup_{k=0}^m\ran(F_k) \quad\Leftrightarrow\quad |\bfy|\leqslant m
\end{equation}
is in $\calA'_{D_m}$. As a particular case, take $F_k=\id_{X^k}$ for $k=0,\ldots,m$ and $|F(\bfx)|>m$ for every $\bfx\in X^*$ such that $|\bfx|>m$. The following example illustrates this case and shows that $\calA_{D_{m+1}}\varsubsetneq\calA_{D_m}$ and $\calA'_{D_{m+1}}\varsubsetneq\calA'_{D_m}$.

\begin{example}\label{ex:2.ex1}
Let $m\geqslant 0$ be an integer and consider the string function $F\colon X^*\to X^*$ that transforms a string of letter into its prefix of length $m$. That is, the $k$-ary part $F_k$ of $F$ is defined by
$$
F_k(\bfx) ~=~ F_k(x_1\cdots{\,}x_k) ~=~
\begin{cases}
\bfx{\,}, & \text{if $k\leqslant m$},\\
x_1\cdots{\,}x_m{\,}, & \text{if $k>m$}.
\end{cases}
$$
It is easy to see that this function is associative. Now, assume $X=\{a,b,\ldots,z\}$ and let $\alpha\colon X\to\{c,v\}$ be defined by $\alpha(x)=v$, if $x$ is a vowel, and $\alpha(x)=c$, if $x$ is a consonant. Let $G\colon X^*\to X^*$ be the ``indexing'' function whose $k$-ary part $G_k$ is defined by
$$
G_k(\bfx) ~=~
\begin{cases}
\bfx{\,}, & \text{if $k\leqslant m$},\\
x_1\cdots{\,}x_m{\,}\alpha(x_{m+1})\cdots{\,}\alpha(x_k){\,}, & \text{if $k>m$}.
\end{cases}
$$
As mentioned above, $G$ is in $\calA'_{D_m}$ and hence in $\calA_{D_m}$. However, it is not in $\calA_{D_{m+1}}$ and hence not in $\calA'_{D_{m+1}}$. Indeed, we have $G(G(a^{m+1}))\neq G(a^{m+1})$.
\end{example}

\begin{proposition}\label{prop:6vdfg6}
Let $m\geqslant 0$ be an integer. If $F\colon X^*\to X^*$ is in $\calA_{D_m}$ and satisfies \eqref{eq:inter4}, then $F$ is in $\calA'_{D_m}$.
\end{proposition}

\begin{proof}
Let $\bfx,\bfy,\bfz \in X^*$ such that $\textstyle{F(\bfy)\in\bigcup_{k=0}^m\ran(F_k)}$. By \eqref{eq:inter4} we must have $|\bfy|\leqslant m$. But then \eqref{eq:assoc} holds (since $F$ is in $\calA_{D_m}$).
\end{proof}

The following example illustrates Proposition~\ref{prop:6vdfg6} and provides a string function in $\calA'_{D_m}$ that does not satisfy \eqref{eq:inter4}.

\begin{example}\label{ex:2.ex2}
Assume $X=\mathbb{L}\cup\N$, where $\mathbb{L}=\{a,b,\ldots,z\}$ and $\N=\{0,1,2,\ldots\}$. For every $\bfx\in X^*$, let $|\bfx|_{\mathbb{L}}$ be the number of letters of $\bfx$ that are in $\mathbb{L}$. Let $m\geqslant 0$ be an integer and consider the functions $F,G\colon X^*\to X^*$ defined by
$$
F(\bfx) ~=~
\begin{cases}
\bfx{\,}, & \text{if $|\bfx|<m$},\\
x_1\cdots{\,}x_{m-1}{\,}|\bfx|{\,}, & \text{if $|\bfx|\geqslant m$},
\end{cases}
$$
and
$$
G(\bfx) ~=~
\begin{cases}
\bfx{\,}, & \text{if $|\bfx|<m$},\\
x_1\cdots{\,}x_m{\,}|\bfx|_{\mathbb{L}}{\,}, & \text{if $|\bfx|\geqslant m$}.
\end{cases}
$$
Clearly, $F$ satisfies \eqref{eq:inter4}. However, $G$ does not since $G(a^m)=G(a^m1)$ for any $a\in\mathbb{L}$.

Let us now prove that both $F$ and $G$ are in $\calA'_{D_m}\setminus\calA_{D_{m+1}}$. By Proposition~\ref{prop:6vdfg6}, to see that $F$ is in $\calA'_{D_m}$ it suffices to show that it is in $\calA_{D_m}$. Let $\bfx,\bfy,\bfz \in X^*$ such that $|\bfy|=m$. We then have
$$
F(\bfx F(\bfy)\bfz) ~=~ F(\bfx{\,} y_1\cdots{\,}y_{m-1}{\,}|\bfy|{\,}\bfz) ~=~ F(\bfx\bfy\bfz),
$$
which shows that $F\in\calA_{D_m}$. Now, $F\notin\calA_{D_{m+1}}$ since $F(F(a^{m+1}))\neq F(a^{m+1})$ for any $a\in\mathbb{L}$. Let us show that $G\in\calA'_{D_m}$. Let $\bfx,\bfy,\bfz \in X^*$ such that $\textstyle{G(\bfy)\in\bigcup_{k=0}^m\ran(G_k)}$. If $G(\bfy)\in\ran(G_k)$ for some $k<m$, then $G(\bfy)=\bfy$ and hence Eq.~\eqref{eq:assoc} clearly holds. If $G(\bfy)\in\ran(G_m)$, then $G(\bfy)=G(y_1\cdots{\,} y_m)$ and hence $|\bfy|_{\mathbb{L}}=|y_1\cdots{\,} y_m|_{\mathbb{L}}$. Moreover, $\bfx\bfy\bfz$ and $\bfx G(\bfy)\bfz$ have the same prefix of length $m$. Therefore Eq.~\eqref{eq:assoc} holds. This shows that $G\in\calA'_{D_m}$. However, we have $G\notin\calA_{D_{m+1}}$ since $G(G(a^{m+1}))\neq G(a^{m+1})$ for any $a\in\mathbb{L}$.
\end{example}

One can easily show that $\calA'_{D_0} \varsubsetneq \calA_{D_0}$. Indeed, take the function $F\colon X^*\to X^*$ defined by $F(\varepsilon)=F(a)=\varepsilon$ for some $a\in X$ and $F(\bfx)=\bfx$ if $\bfx\notin\{\varepsilon,a\}$. The following example shows that $\calA'_{D_m} \varsubsetneq \calA_{D_m}$ for every integer $m\geqslant 1$.

\begin{example}\label{ex:e675}
Let $m\geqslant 1$ be an integer and consider the function $F\colon X^*\to X^*$ defined by
$$
F(\bfx) ~=~
\begin{cases}
\bfx{\,}, & \text{if $|\bfx|\leqslant m$ or $|\bfx|=m+2$},\\
x_1\cdots{\,}x_{m}{\,}, & \text{if $|\bfx|=m+1$ or $|\bfx|>m+2$}.
\end{cases}
$$
Clearly $F$ is in $\calA_{D_m}$. However, it is not in $\calA'_{D_m}$ since, setting $\bfy=a^{m+1}$ and $z=a$ for some $a\in X$, we have $F(\bfy)=a^m=F(a^m)\in\ran(F_m)$ but $F(\bfy z)=a^{m+2}\neq a^m=F(F(\bfy)z)$.
\end{example}

Just as we have $\calA\subseteq\calP$, we also have $\calA_{D_m}\subseteq\calP_{D_m}$ and $\calA'_{D_m}\subseteq\calP'_{D_m}$ for every integer $m\geqslant 0$. This observation immediately follows from Propositions~\ref{prop:wer61} and \ref{prop:wer6} below. Let us now show that these inclusions are strict. For $m=0$, take $F\colon X^*\to X^*$ such that $F(\varepsilon)=a$ for some $a\in X$ and $F(\bfx)=\bfx$ for every $\bfx\neq\varepsilon$. Then $F$ is in $\calP_{D_0}\setminus \calA_{D_0}$. For $m\geqslant 1$, let $\sigma$ be a nontrivial permutation on $X$. The function $F\colon X^*\to X^*$ defined by $F(\varepsilon)=\varepsilon$ and $F(x_1\cdots{\,}x_n)=\sigma(x_1)\cdots{\,}\sigma(x_n)$ for every integer $n\geqslant 1$ is in $\calP\setminus\calA_{D_1}$. Consider for instance in Eq.~\eqref{eq:assoc} the strings $\bfx=\bfz=\varepsilon$ and $\bfy =\sigma^{-1}(a)$ for some $a\in X$ such that $\sigma(a)\neq a$.

Let us now show that the sets $\calP_{D_m}\setminus\calP_{D_{m+1}}$ and $\calP'_{D_m}\setminus\calP'_{D_{m+1}}$ are nonempty. Consider first the case $m=0$. Take $X=\R$ and the function $F\colon\R^*\to\R^*$ defined by $F(\varepsilon)=\varepsilon$ and $F(x_1\cdots{\,}x_n)=\frac{1}{n}\sum_{i=1}^nx_i$ for every integer $n\geqslant 1$. Then $F$ is in $\calP'_{D_0}$ (and hence in $\calP_{D_0}$) but not in $\calP_{D_1}$ (and hence not in $\calP'_{D_1}$). Indeed, we have $F(0)=F(00)$ but $F(01)\neq F(001)$. For the case $m\geqslant 1$, take $F\colon X^*\to X^*$ defined as
$$
F(\bfx) ~=~
\begin{cases}
\bfx{\,}, & \text{if $|\bfx|\leqslant m$},\\
x_1\cdots{\,}x_{m}x_{k-1}{\,}, & \text{if $|\bfx|=k>m$}.
\end{cases}
$$
Then, $F$ is in $\calP'_{D_m}$ but not in $\calP_{D_{m+1}}$. Indeed, if $z\neq z'$, we have
$$
F(y_1\cdots{\,}y_m z) ~=~ F(y_1\cdots y_m z')
$$
but
$$
F(y_1\cdots{\,}y_m za) ~\neq ~ F(y_1\cdots y_m z'a).
$$

\section{Characterization results}

In this section we aim at localizing each of the parameterized classes $\calA_D$ introduced above within its corresponding superclass $\calP_D$. This goal is achieved in two ways: on the one hand, under the assumption that $F(D)\subseteq D$ we show that functions in $\calA_D$ are exactly those in $\calP_D$ that verify the condition $F|_D=F\circ F|_D$ and, on the other hand, we also show that functions in $\calP_D$ admit factorizations in terms of functions in $\calA_D$. Similar results are then established for $\calA'_D$ and $\calP'_D$.

As observed we have $\calA\subseteq\calP$ and this inclusion is actually strict. In fact, we have the following result.

\begin{proposition}[{\cite{MarTeh14,MarTeh14b}}]\label{prop:wer60}
A function $F\colon X^*\to X^*$ is in $\calA$ if and only if it is in $\calP$ and satisfies $F=F\circ F$.
\end{proposition}

The following two propositions provide generalizations of Proposition~\ref{prop:wer60} to the classes $\calA_D$, $\calA'_D$, $\calP_D$, and $\calP'_D$.

\begin{proposition}\label{prop:wer61}
Let $D$ be an nonempty subset of $X^*$ and consider a function $F\colon X^*\to\ran(F)$. If $F$ is in $\calA_D$, then it is in $\calP_D$ and satisfies $F|_D=F\circ F|_D$. The converse also holds whenever $F(D)\subseteq D$.
\end{proposition}

\begin{proof}
Suppose $F\in\calA_D$. Clearly, we have $F|_D=F\circ F|_D$. Now, let $\bfy,\bfy'\in D$ be such that $F(\bfy)=F(\bfy')$. Then, we have
$$
F(\bfx\bfy\bfz) ~=~ F(\bfx F(\bfy)\bfz) ~=~ F(\bfx F(\bfy')\bfz) ~=~ F(\bfx\bfy'\bfz),
$$
which shows that $F\in\calP_D$. For the converse statement, let $\bfx,\bfz\in X^*$ and $\bfy\in D$. We then have $F(\bfy)=F(F(\bfy))$ and hence, since $F$ is in $\calP_D$, we also have $F(\bfx\bfy\bfz)=F(\bfx F(\bfy)\bfz)$, which shows that $F\in\calA_D$.
\end{proof}

\begin{proposition}\label{prop:wer6}
Let $D$ be a nonempty subset of $X^*$ and let $F\colon X^*\to\ran(F)$ be such that $F(D)\subseteq X^*$. Then $F$ is in $\calA'_D$ if and only if it is in $\calP'_D$ and satisfies $F|_D=F\circ F|_D$.
\end{proposition}

\begin{proof}
(Necessity) Let $\bfy\in D$ and $\bfy'\in X^*$ be such that $F(\bfy)=F(\bfy')$. Since $F\in\calA'_D$, for every $\bfx,\bfz\in X^*$ we have
$$
F(\bfx\bfy\bfz) ~=~ F(\bfx F(\bfy)\bfz) ~=~ F(\bfx F(\bfy')\bfz) ~=~ F(\bfx\bfy'\bfz),
$$
which shows that $F\in\calP'_D$. Also, we clearly have $F|_D=F\circ F|_D$.

(Sufficiency) Let $\bfx,\bfz\in X^*$ and $\bfy\in D$. As $F(\bfy)=F(F(\bfy))$ and $F\in\calP'_D$, we also have $F(\bfx\bfy\bfz)=F(\bfx F(\bfy)\bfz)$.

Now, let $\bfy'\in X^*$ such that $F(\bfy')=F(\bfy)$. Since $F\in\calP'_D$ we have
$$
F(\bfx\bfy'\bfz)=F(\bfx\bfy\bfz)=F(\bfx F(\bfy)\bfz)=F(\bfx F(\bfy')\bfz),
$$
which shows that $F\in\calA'_D$.
\end{proof}

Let us recall the following factorization established in \cite{LehMarTeh14}.

\begin{theorem}[{\cite{LehMarTeh14}}]\label{thm:ds8610}
Assume AC and let $F\colon X^*\to Y$ be a function. The following assertions are equivalent.
\begin{enumerate}
\item[(i)] $F\in\calP$.

\item[(ii)] There exist $H\in\calA$ and a one-to-one function $f\colon\ran(H)\to Y$ such that $F=f\circ H$.
\end{enumerate}
For any $g\in Q(F)$, we can choose $H=g\circ F$ and $f=F|_{\ran(H)}$ in assertion (ii).
\end{theorem}

We will now generalize Theorem~\ref{thm:ds8610} to the classes $\calP_D$ and $\calP'_D$ (see Theorems~\ref{thm:ds861} and \ref{thm:ds86}). For this purpose, we first present a more general factorization result (see Proposition~\ref{prop:1}).

Let $\mathcal{C}$ be any class of functions defined on a set $\Omega$ and satisfying the following property: if $F\in\mathcal{C}$, then $g\circ F\in\mathcal{C}$ for every one-to-one map $g$ defined on $\ran(F)$. For any nonempty set $D\subseteq\Omega$, define also the following subclasses:
\begin{eqnarray*}
\mathcal{C}'_D &=& \{F\in\mathcal{C}\mid F(D)\subseteq D\}{\,},\\
\mathcal{C}''_D &=& \{F\in\mathcal{C}\mid F(D)\subseteq D~\text{and}~F\circ F|_D=F|_D\}{\,}.
\end{eqnarray*}

\begin{example}
If $\mathcal{C}$ is the class $\calP$ of preassociative functions on $X^*$, then $\mathcal{C}''_{X^*}$ is the class $\calA$ of associative string functions on $X^*$ by Proposition~\ref{prop:wer60}.
\end{example}

We also observe that, assuming AC, for any subset $D$ of $\dom(f)$ the set
$$
Q_D(f) ~=~ \{g\in Q(f)\mid (g\circ f)(D)\subseteq D\}
$$
is nonempty. Indeed, if $x\in D$, we can set $g(f(x))=x$.

\begin{proposition}\label{prop:1}
Assume AC, let $D\subseteq\Omega$, and let $F$ be a function defined on $\Omega$. The following assertions are equivalent.
\begin{enumerate}
\item[(i)] $F\in\mathcal{C}$.

\item[(ii)] There exist $H\in\mathcal{C}'_D$ and a one-to-one function $f$ defined on $\ran(H)$ such that $F=f\circ H$.

\item[(iii)] There exist $H\in\mathcal{C}''_D$ and a one-to-one function $f$ defined on $\ran(H)$ such that $F=f\circ H$.
\end{enumerate}
For any $g\in Q_D(F)$, we can choose $H=g\circ F$ and $f=F|_{\ran(H)}$ in assertions (ii) and (iii).
\end{proposition}

\begin{proof}
(iii) $\Rightarrow$ (ii) Trivial.

(ii) $\Rightarrow$ (i) Since $H\in\mathcal{C}'_D\subseteq\mathcal{C}$, we have $F\in\mathcal{C}$.

(i) $\Rightarrow$ (iii) Let $g\in Q_D(F)$ and set $H=g\circ F$. Since $g|_{\ran(F)}$ is one-to-one, we have $H\in\mathcal{C}$. Also, we have $H(D)=(g\circ F)(D)\subseteq D$ and $H\circ H|_D=g\circ F\circ g\circ F|_D=g\circ F|_D=H|_D$. It follows that $H\in\mathcal{C}''_D$.

Now, let $f=F|_{\ran(H)}$. Since $\ran(H)=\ran(g\circ F)=\ran(g)$, the map $f$ is one-to-one. Finally, we have $f\circ H=F\circ H=F\circ g\circ F = F$.
\end{proof}

Using Proposition~\ref{prop:1} we can now derive the following two generalizations of Theorem~\ref{thm:ds8610}.

\begin{theorem}\label{thm:ds861}
Assume AC, let $D\subseteq X^*$, and let $F\colon X^*\to Y$ be a function. The following assertions are equivalent.
\begin{enumerate}
\item[(i)] $F\in\calP_D$.

\item[(ii)] There exist $H\in\calA_D$ and a one-to-one function $f\colon\ran(H)\to Y$ such that $F=f\circ H$.

\item[(iii)] There exist $H\in\calA_D$ such that $H(D)\subseteq D$ and a one-to-one function $f\colon\ran(H)\to Y$ such that $F=f\circ H$.
\end{enumerate}
For any $g\in Q_D(F)$, we can choose $H=g\circ F$ and $f=F|_{\ran(H)}$ in assertions (ii) and (iii).
\end{theorem}

\begin{proof}
(i) $\Leftrightarrow$ (iii) Setting $\Omega=X^*$ and $\mathcal{C} = \calP_D$, we have
$$
\mathcal{C}''_{D} ~=~ \{H\colon X^*\to\ran(H)\mid\text{$H\in\calP_D$, $H(D)\subseteq D$, and $H\circ H|_D=H|_D$}\}.
$$
By Proposition~\ref{prop:wer61}, we also have
$$
\mathcal{C}''_{D} ~=~ \{H\colon X^*\to\ran(H)\mid\text{$H\in\calA_D$ and $H(D)\subseteq D$}\}.
$$
We conclude the proof by making use of the equivalence between (i) and (iii) of Proposition~\ref{prop:1}.

(iii) $\Rightarrow$ (ii) $\Rightarrow$ (i) Straightforward. 
\end{proof}

\begin{theorem}\label{thm:ds86}
Assume AC, let $D\subseteq X^*$, and let $F\colon X^*\to Y$ be a function. The following assertions are equivalent.
\begin{enumerate}
\item[(i)] $F\in\calP'_D$.

\item[(ii)] There exist $H\in\calA'_D$ and a one-to-one function $f\colon\ran(H)\to Y$ such that $F=f\circ H$.

\item[(iii)] There exist $H\in\calA'_D$ satisfying $H(D)\subseteq D$ and a one-to-one function $f\colon\ran(H)\to Y$ such that $F=f\circ H$.

\item[(iv)] There exist $H\in\calA'_D$ satisfying $\ran(H)\subseteq X^*$ and $H\circ H=H$ and a one-to-one function $f\colon\ran(H)\to Y$ such that $F=f\circ H$.
\end{enumerate}
For any $g\in Q(F)$ (resp.\ $g\in Q_D(F)$), we can choose $H=g\circ F$ and $f=F|_{\ran(H)}$ in assertions (ii) and (iv) (resp.\ assertions (ii) and (iii)).
\end{theorem}

\begin{proof}
(i) $\Leftrightarrow$ (iii) The proof follows from Propositions~\ref{prop:wer6} and \ref{prop:1} with $\mathcal{C} = \calP'_D$ and considering $\mathcal{C}''_D$.

(i) $\Leftrightarrow$ (iv) The proof follows from Propositions~\ref{prop:wer6} and \ref{prop:1} with $\mathcal{C} = \calP'_D$ and considering $\mathcal{C}''_{X^*}$.

(iv) $\Rightarrow$ (ii) $\Rightarrow$ (i) Straightforward.
\end{proof}

We observe that Theorem~\ref{thm:ds8610} is a consequence of Theorem~\ref{thm:ds861} (or, equivalently, of Theorem~\ref{thm:ds86}) whenever $D=X^*$.

\section{Functions having a $D$-determined range}

We now turn our attention to classes of functions with a prescribed range. Recall that, for every integer $m\geqslant 0$, a function $F \colon X^* \to X^*$ is said to be \emph{$m$-bounded} if $|F(\bfx)|\leqslant m$ for every $\bfx\in X^*$. A function $F\colon X^*\to Y$ is said to have an \emph{$m$-determined range} if $\ran(F)=\bigcup_{k=0}^m\ran(F_k)$ (see \cite{LehMarTeh14}). These concepts can be generalized in the following way.

\begin{definition}
Let $D$ be a nonempty subset of $X^*$. We say that a map $F\colon X^*\to Y$
\begin{itemize}
\item is \emph{$D$-valued} if $\ran(F)\subseteq D$.

\item has a \emph{$D$-determined range} if $\ran(F)=F(D)$.
\end{itemize}
\end{definition}

Note that the property of having a $D$-determined range is preserved under left composition with unary maps: if $F\colon X^*\to Y$ has a $D$-determined range, then so has $g\circ F$ for any map $g\colon Y\to Y'$, where $Y'$ is an nonempty set.

\begin{lemma}\label{fact:8dsf}
Let $D$ be a nonempty subset of $X^*$.
\begin{itemize}
\item[(a)] If $F\colon X^*\to X^*$ is $D$-valued and satisfies $F=F\circ F$, then $F$ has a $D$-determined range.
\item[(b)] If $F\colon X^*\to Y$ has a $D$-determined range and satisfies $F(D)\subseteq X^*$ and $F|_D=F\circ F|_D$, then it satisfies $\ran(F)\subseteq X^*$ and $F=F\circ F$.
\item[(c)] If $F\colon X^*\to Y$ satisfies $F=F\circ H$, where $H\colon X^*\to X^*$ is $D$-valued, then it has a $D$-determined range.
\item[(d)] If $F\colon X^*\to X^*$ is in $\calA'_D$, $D$-valued, and satisfies $F=F\circ F$, then it is associative.
\item[(e)] If $F\colon X^*\to Y$ is in $\calA'_D$ and has a $D$-determined range, then it is associative.
\item[(f)] If $F\colon X^*\to Y$ is in $\calP'_D$ and has a $D$-determined range, then it is preassociative.
\end{itemize}
\end{lemma}

\begin{proof}
The proofs of statements (a)--(e) are straightforward. To see that (f) holds, let $\bfy,\bfy'\in X^*$ such that $F(\bfy)=F(\bfy')$. Since $F$ has a $D$-determined range, there exists $\bfu\in D$ such that $F(\bfy)=F(\bfu)=F(\bfy')$. Since $F\in\calP'_D$, we thus have $F(\bfx\bfy\bfz)=F(\bfx\bfu\bfz)=F(\bfx\bfy'\bfz)$, which shows that $F$ is preassociative.
\end{proof}

\begin{theorem}
Assume AC, let $D\subseteq X^*$, and let $F\colon X^*\to Y$ be a function. The following assertions are equivalent.
\begin{enumerate}
\item[(i)] $F$ is preassociative and has a $D$-determined range.

\item[(ii)] $F$ is in $\calP'_D$ and has a $D$-determined range.

\item[(iii)] There exist an associative and $D$-valued function $H\colon X^*\to X^*$ and a one-to-one function $f\colon\ran(H)\to Y$ such that $F=f\circ H$.

\item[(iv)] There exist a function $H\in\calA'_D$ with a $D$-determined range 
    and a one-to-one function $f\colon\ran(H)\to Y$ such that $F=f\circ H$.

\item[(v)] There exist an associative function $H\colon X^*\to X^*$ with a $D$-determined range and a one-to-one function $f\colon\ran(H)\to Y$ such that $F=f\circ H$.
\end{enumerate}
\end{theorem}

\begin{proof}
The equivalence between (i) and (iii) follows from Proposition~\ref{prop:1}, where $\mathcal{C}$ is the class of preassociative functions on $X^*$ that have a $D$-determined range. Indeed, by Lemma~\ref{fact:8dsf}(a), Lemma~\ref{fact:8dsf}(b) and Proposition~\ref{prop:wer60} the class $\mathcal{C}''_D$ then consists of the $D$-valued associative functions on $X^*$. The equivalence between (i) and (ii) follows from Lemma~\ref{fact:8dsf}(f). The equivalence between (ii) and (iv) follows from Theorem~\ref{thm:ds86} and the fact that the property of having a $D$-determined range is preserved under left composition. The implication (iv) $\Rightarrow$ (v) follows from Lemma~\ref{fact:8dsf}(e). The implication (v) $\Rightarrow$ (iv) is trivial.
\end{proof}

As far as the sets $D_m$ ($m\geqslant 0$) are concerned, we also have the following result, which provides an alternative condition for a function in $\calP'_{D_m}$ to have an $m$-determined range.

\begin{proposition}
Let $F\colon X^*\to Y$ be in $\calP'_{D_m}$. The following assertions are equivalent:
\begin{enumerate}
\item[(i)] $\ran(F_{m+1})\subseteq\bigcup_{k=0}^m\ran(F_k)$

\item[(ii)] $\ran(F)=\bigcup_{k=0}^m\ran(F_k)$ ($F$ has an $m$-determined range).
\end{enumerate}
\end{proposition}

\begin{proof}
(ii) $\Rightarrow$ (i) Trivial (preassociativity is not needed).

(i) $\Rightarrow$ (ii) We only need to show that $\ran(F)\subseteq\bigcup_{k=0}^m\ran(F_k)$. Let $F(\bfx)\in\ran(F)$. We have to consider the following mutually exclusive cases:
\begin{enumerate}
  \item[(a)] If $|\bfx|\leqslant m$, then we are done since $F(\bfx)\in\bigcup_{k=0}^m\ran(F_k)$.
  \item[(b)] If $|\bfx|=k\geqslant m+1$, then by (i) there exists $\bfu\in X^*$, with $|\bfu|\leqslant m$, such that $F(x_1\cdots x_{m+1})=F(\bfu)$. Since $f\in \calP'_{D_m}$, we have $F(\bfx)=F(\bfy)$, where $\bfy=\bfu{\,}{x_{m+2}\cdots x_k}$. Since $|\bfy|<|\bfx|$, we can iterate the process and then we are done after at most $|\bfx|-m$ iterations.\qedhere
\end{enumerate}
\end{proof}

\section{Conclusion and further research}

The properties of associativity and preassociativity for functions defined over strings are given in terms of a functional equation and a logical implication, respectively. In this paper we relaxed these properties by imposing restrictions on the variables of these defining conditions.

In particular, in Section 2 we showed that certain restrictions on the length of the string variables induce strict hierarchies of nested classes whose intersections reduce to the classes of associative and preassociative functions, respectively. Apart from the theoretical interest, such hierarchies can be used to measure degrees of associativeness (resp.\ preassociativeness). Indeed, by setting $d(f)=2^{-k}$ where $k$ is the minimum positive integer $m$ such that $f\in \calA_{D_m}\setminus\calA_{D_{m+1}}$ (resp.\ $f\in \calP_{D_m}\setminus\calP_{D_{m+1}}$), we see that $d$ measures how distant $f$ is from being associative (resp.\ preassociative).

In Section 3, for each nonempty $D\subseteq X^*$, we provided additional conditions that reduce classes $\calP_D$ and $\calP'_D$ to
$\calA_D$ and $\calA'_D$, respectively. These results were then complemented by factorizations of $\calP_D$ and $\calP'_D$ into composites $\calI\circ\calA_D$ and $\calI\circ\calA'_D$, where $\calI$ is a class of one-to-one functions. These factorization results may be particularly useful. Indeed, they enable us to construct functions in $\calP_D$ (resp.\ $\calP'_D$) from known functions in $\calA_D$ (resp.\ $\calA'_D$). Also, they may enable us to obtain new axiomatizations of subclasses of $\calP_D$ (resp.\ $\calP'_D$) from existing axiomatizations of subclasses of $\calA_D$ (resp.\ $\calA'_D$). This observation was already fruitfully used for $D=X^*$ (see \cite{MarTeh15}).

Regarding the idea of restricting the variables, alternative natural variants of associativity and preassociativity are to be considered. For instance, for any nonempty subset $D$ of $X^*$, consider the class of functions
$$
\calA^0_D ~=~ \{F\colon D\to X^*\mid\text{\eqref{eq:assoc} holds for all $\bfx,\bfy,\bfz \in X^*$ such that $\bfx\bfy\bfz,{\,}\bfy,{\,}\bfx F(\bfy)\bfz\in D$}\}.
$$
It is clear that if $F\in\calA_D$, then $F|_D\in\calA^0_D$. However, the converse is an open question of interest: give necessary and sufficient conditions on a nonempty subset $D\subseteq X^*$ and a function $F\in\calA^0_D$ for the existence of an extension $G\in\calA_D$ (i.e., $G|_D=F$). For example, if $\bfx D\bfz\subseteq D$ for all $\bfx,\bfz \in X^*$ and $F(D)\subseteq D$, then the function $G$ defined by $G|_D=F$ and $G|_{X^*\setminus D}=\id_{X^*\setminus D}$ is in $\calA_D$.

The same question can be addressed for the class
$$
\calP^0_D ~=~ \{F\colon D\to\ran(F)\mid\text{\eqref{eq:preassoc} holds for all $\bfx,\bfy,\bfy',\bfz \in X^*$ such that $\bfx\bfy\bfz,{\,}\bfx\bfy'\bfz,{\,}\bfy,{\,}\bfy'\in D$}\}.
$$

These questions constitute topics of ongoing research work.

%
%

\end{document}